\documentclass[11pt]{article}
\usepackage{amsmath}
\usepackage{amsthm}
\usepackage{amssymb}

\newtheorem{thm}{Theorem}
\newtheorem{cor}{Corollary}
\newtheorem{proposition}{Proposition}

\newtheorem{remark}{Remark}

\begin{document}

\vspace*{50px}

\begin{center}\LARGE
%%% Insert the title of your talk here %%%

\textbf{Some Families of Identities for Integer Partition Function}

\bigskip\large

%%% Insert your first name and family name here %%%
Ivica Martinjak, Dragutin Svrtan

%%% Insert your institution here %%%
University of Zagreb

%%% Insert you city and country here %%%
Zagreb, Croatia

%%% Insert your institution here %%%

%%% Insert you city and country here %%%
\end{center}

%%% Give the abstract of your talk here %%% 
 
\begin{abstract}  
We give a series of recursive identities for the number of partitions with exactly $k$ parts and with  constraints on both the minimal difference among the parts and the minimal part. Using these results we demonstrate that the number of partitions of $n$ is equal to the number of partitions of $2n+d{n \choose 2}$ with $n$ $d$-distant parts. We also provide a direct proof for this identity. This work is the result of our aim at finding a bijective proof for Rogers-Ramanujan identities.
\end{abstract} 
%\maketitle
%\tableofcontents


%%%% **** The text of the paper starts here **** %%%%

\section{Introduction}

The sequence of non-negative integers $\lambda=(\lambda_1,\lambda_2,...,\lambda_l)$ in weakly decreasing order is called a {\it partition}. The numbers $\lambda_i$ are {\it parts} of $\lambda$. The number of parts is the {\it length} of $\lambda$, denoted by $l(\lambda)$ and the sum of parts $|\lambda|$ is the {\it weight} of $\lambda$. Having $|\lambda|=n$ it is said that $\lambda$ is a partition of $n$, denoted $\lambda \vdash n$.

The set of all partitions is denoted by ${\mathcal{P}}$ and the set of partitions of $n$ by ${\mathcal{P}}_n$. The number of elements in ${\mathcal{P}}_n$ is denoted by $p(n)$. The number of partitions of $n$ with exactly $k$ parts is denoted by $p_k(n)$ while $p_{\le k}(n)$ is the number of partitions of $n$ with at most $k$ parts $p_{\le k}(n)=\sum_{i=0}^k p_i(n).$ The number of partitions of $n$ having minimal part at least $r$ is denoted by $p(n,r)$. Furthermore, of our interests are partitions with distant parts. Let $p^{(1)}(n)$ be the number of partitions of $n$ with the characteristic that the difference between any two parts is at least $1$. The number of partitions with $2$-distant parts, that represent the left hand side of the first Rogers-Ramanujan identity, we denote by $p^{(2)}(n)$. It is also said that such partitions have super-distant parts. Accordingly, $p^{(d)}(n)$ is the number of partitions of $n$ with $d$-distant parts.

It is known that the {\it Euler identity} 
\begin{equation}
\phi(x)P(x)=1
\end{equation}
provides a recursive computation of the numbers $p(n)$. Namely, setting the {\it Euler function} $\phi(x)$ and the {\it partition function} $p(x)$ in the previous identity we get
\begin{equation*}
(1-x-x^2+x^5+x^7 -x^{12}...)(1+p(1)x+p(2)x^2+p(3)x^3+...)=1.
\end{equation*}
Since the coefficient of $x^n$, $n>0$ in this product is equal to $0$, we obtain the recurrence relation
\begin{equation}
p(n)=p(n-1)+p(n-2)-p(n-5)-p(n-7)+p(n-12)...
\end{equation}
where the general terms include the pentagonal number
$$\frac{j(3j-1)}{2}$$
and its $j$-th successor
$$\frac{j(3j+1)}{2}.$$

There is also a recurrence relation that provides computing the number of partitions with exactly $k$ parts
\begin{equation}\label{psThm1}
p_k(n)=p_{k-1}(n-1)+p_k(n-k).
\end{equation}
Namely, a partition in the set of partitions $\lambda \vdash n$ with $k$ parts either has number 1 as a part or not. The latter means that subtracting the partition for this part 1 gives a partition from the set ${\mathcal{P}}_{n-1}$ with $k$--1 parts. On the other hand, when the minimal part in the partition is equal or greater than $2$, subtracting $1$ from every part of this partition leave the number of parts unchanged while the weight of partition decreases by $k$. This explains the above recurrence relation.

In this paper we extend these ideas, searching for similar recurrences for other types of partitions. In particular, we are interested in partitions of length $k$ having 1-distant parts, 2-distant parts, etc. Underlying motivation for this work is to find a new bijective proof of Rogers-Ramanujan identities. Namely, recurrences for both l.h.s and r.h.s. of the identity possibly give an insight into the matching of related partitions. Recall that the first bijective proof is done by Garsia and Milne while the recent one is provided by Pak and Boulet \cite{BoPa}.

It is worth mentioning that an efficient way of computing the number of partitions of $n$ with at most $k$ parts is provided by partial fraction decomposition of the generating function. This idea dates back to Cayley, it is developed by Munagi \cite{Mun} and formulae for $1 \le k \le 70$ are recently derived by Sills and Zeilberger \cite{SiZe}.

\section{Partitions of length $k$ with $1$-distant parts}

\begin{proposition}\label{Lem1} 
The number of partitions $\lambda \vdash n$ of length $k$ with 1-distant parts and with minimal part at least 2 is equal to the alternating sum of numbers of partitions $\mu_i \vdash n-i, \enspace l(\mu_i)=k-i$, $ 0 \leq i  \leq k$ with $1$-distant parts
\begin{equation}
p^{(1)}_k(n,2)=\sum_{i=0}^{k}(-1)^i p_{k-i}^{(1)}(n-i).
\end{equation}
\end{proposition}

\begin{proof} 
The number of partitions of $n$ of length $k$ and with $1$-distant parts and with the minimal part at least 2 equals the difference
\begin{equation}
p^{(1)}_k(n,2)=p^{(1)}_k(n)-p^{(1)}_{k-1}(n-1,2).
\end{equation}
Namely, if we add the number 1 as a part to the partition of $n$--1 of length $k-1$ of with 1-distant parts and the mininal part at least 2, the resulting partition will be $k$-length partition of $n$ with saved 1-distant propositionerty and with the minimal part exactly 1. Now the above statement follows immediately. Furthermore, we have
\begin{equation*}
p^{(1)}_k(n,2)=p^{(1)}_k(n)-p^{(1)}_{k-1}(n-1)+p^{(1)}_{k-2}(n-2,2)
\end{equation*}
which completes the statement of proposition.
\end{proof}

There is an analogy with the relation \ref{psThm1} for the partitions with $1$-distant parts and it is given by the next proposition.

\begin{proposition}\label{Thm1}
The number of partitions $\lambda \vdash n$ of length $k$ with $1$-distant parts is equal to the sum of the numbers of partitions $\mu \vdash n-1$ of length k--1 with 1-distant parts having the minimal part equal or greater than 2 and the number of partitions $\nu \vdash n-k$ of length $k$ with 1-distant parts 
\begin{equation}
p^{(1)}_k(n)=p^{(1)}_{k-1}(n-1,2)+p^{(1)}_k(n-k).
\end{equation}
\end{proposition}

\begin{proof}
Let separate the partitions of $n$ with $k$ $1$-distant parts into two sets, one that has the number 1 as a part and another one with parts equal to or greater than 2. The first set can be built from partitions of length $k$--1 of the number $n$--1 with $1$-distant parts greater of or equal to 2, by adding 1 as a part. The second set we obtain from the set of all partitions of $n$--k of length $k$ with 1-distant parts by increasing every part by 1. Obviously, these correspondences are invertible which completes the proof.
\end{proof}

Using the Propositions \ref{Lem1} and \ref{Thm1} we will prove the recurrence relations for the numbers $p_k^{(1)}(n)$, expressed in the next corollary. Later we will point out to a shorter proof too.

\begin{cor}\label{Cor1}  
The number of partitions $\lambda \vdash n$ with $k$ 1-distant parts is equal to the sum of numbers of partitions $\mu_i \vdash n-ik$, $i=1,2,...$ with k--1 $1$-distant parts
\begin{equation}
p_k^{(1)}(n) = \sum_{i \ge 1} p_{k-1}^{(1)}(n-ki). \label{Recu2}
\end{equation}
\end{cor}

\begin{proof}
The proof is done in the three phases. Firstly we show that the relation 
\begin{equation}\label{Rel10}
p_k^{(1)}(n)=p_{k-1}^{(1)}(n-1,2)+p_{k-1}^{(1)}(n-k-1,2)+\ldots+p_{k-1}^{(1)}(n-qk-1,2) 
\end{equation}
is the consequence of the previous proposition. It holds
\begin{eqnarray*}
p_k^{(1)}(n)&=&p_{k-1}^{(1)}(n-1,2)+p_k^{(1)}(n-k)\\
p_k^{(1)}(n-k)&=&p_{k-1}^{(1)}(n-k-1,2)+p_k^{(1)}(n-2k)\\
\ldots\\
p_k^{(1)}(n-qk)&=& p_{k-1}^{(1)}(n-qk-1,2)+p_k^{(1)}(n-(q+1)k)
\end{eqnarray*}
where $q \in {\mathbb N}$. Since the argument in the second term in the last equality is less than $k$ the value of the term is 0 and the previous statement follows immediately. Now, using the Proposition \ref{Lem1} every term in the relation \ref{Rel10} can be represented as an alternative sum, which leads to the next array of numbers whose sum equals the number $p_k^{(1)}(n)$.
\begin{eqnarray}\label{TheMainArray}
\begin{array}{rrrrr} 
p_k^{(1)}(n)=&+p_{k-1}^{(1)}(n-1) & +p_{k-1}^{(1)}(n-k-1) &\ldots&+p_{k-1}^{(1)}(n-qk-1) \\ 
&-p_{k-2}^{(1)}(n-2)& - p_{k-1}^{(1)}(n-k-2) &\ldots&-p_{k-1}^{(1)}(n-qk-2) \\
&\ldots\\
&...p_{0}^{(1)}(n-k)& ...p_{k-1}^{(1)}(n-2k) &\ldots&...p_{k-1}^{(1)}(n-qk-2)
\end{array}
\end{eqnarray}
Finally, we apply Propositions \ref{Thm1} and \ref{Lem1} to every term in the first row of this array. Again we obtain similar arrays, starting with the terms $p_{k-2}^{(1)}(n-2)$,  $p_{k-2}^{(1)}(n-k-2),\ldots,p_{k-2}^{(1)}(n-qk-2)$. One can convince ourself that the first rows in these arrays cancel all remaining terms on the right hand side of the relation \ref{TheMainArray} and we get arrays beggining with $p_{k-2}^{(1)}(n-k-1)$,  $p_{k-2}^{(1)}(n-2k-1),\ldots,p_{k-2}^{(1)}(n-qk-k-1)$. The sum of elements of these arrays corresponds to the numbers $p_{k-1}^{(1)}(n-k), p_{k-1}^{(1)}(n-2k),\ldots,p_{k-1}^{(1)}(n-(q+1)k)$ and we obtain the relation
\begin{equation*}
p_{k}^{(1)}(n)=p_{k-1}^{(1)}(n-k)+p_{k-1}^{(1)}(n-2k)+\ldots+p_{k-1}^{(1)}(n-(q+1)k).
\end{equation*}
This completes the statements of the corollary.
\end{proof}

In particular, it follows from the Corollary \ref{Cor1}:
\begin{eqnarray*}
p_4^{(1)}(n)&=&\sum_{i \ge 1} p_3^{(1)}(n-4i)\\
          &=&\sum_{i,j \ge 1} p_2^{(1)}(n-3i-4j)\\
          &=&\sum_{i,j,l \ge 1} p_1^{(1)}(n-2i-3j-4l).
\end{eqnarray*}
In order to illustrate these results we are going to calculate the number of partitions of $22$ with four $1$-distant parts. According to the previous corollary, we have
\begin{eqnarray*}
p^{(1)}_4(22)&=&p^{(1)}_3(18)+p^{(1)}_3(14)+p^{(1)}_3(10)+p^{(1)}_3(6)+p^{(1)}_3(2)\\
%&=&19+10+4+1\\
&=&34.
\end{eqnarray*}
On the other hand, we can represent $p^{(1)}_4(22)$ as the sum of the number of partitions of length 2:
\begin{eqnarray*}
p^{(1)}_4(22)&=&p^{(1)}_2(15)+p^{(1)}_2(12)+p^{(1)}_2(9)+p^{(1)}_2(6)+p^{(1)}_2(3)\\
&+&p^{(1)}_2(11)+p^{(1)}_2(8)+p^{(1)}_2(5)+p^{(1)}_2(2)\\
&+&p^{(1)}_2(7)+p^{(1)}_2(4)+p^{(1)}_2(1)\\
&+&p^{(1)}_2(3)\\
%&=&7+5+4+2+1+5+3+2+3+1+1\\
&=&34.
\end{eqnarray*}
In the third case 39 terms result, five of them being zero.

\section{Partitions of the Rogers-Ramanujan type of fixed length}

The previous facts for partitions with 1-distant parts we can extend to partitions with any difference $d$ among the parts. 

Adding $d$ to the every of $k$ parts of a partition of $n$, we obtain a partition of $n+dk$ of length $k$ with saved difference among the parts. The additional characteristic of the resulting partition is that the minimal part increases to at least $d+1$. Since the inverse operation gives the starting partition, we proved the next proposition.

\begin{proposition}\label{Lem3} 
The number of partitions $\lambda \vdash n$ of length $k$ with $d$-distant parts is equal to the number of partitions $\mu \vdash n+dk$ of length $k$  with $d$-distant parts and the smallest part d+1
\begin{equation}
p^{(d)}_k(n)=p^{(d)}_k(n+dk,d+1).
\end{equation}
\end{proposition}

The next proposition is worth for any difference at least $d$ amoung parts of a partition. It can be proved applying the analogue reasoning as for Proposition \ref{Thm1}. Similarly, the following corollary can be understand as the generalization of the previous corollary, that holds for any difference $d$ amoung the parts of a partition.

\begin{proposition}\label{Thm3} 
The number of partitions $\lambda \vdash n$ of length $k$ with d-distant parts is equal to the sum of number of partitions $\mu \vdash n-1$ with d-distant parts that are equal to or greater than $d+1$ and the number of partitions $\nu \vdash n-k$ of length $k$ with d-distant parts
\begin{equation}
p^{(d)}_k(n)=p^{(d)}_{k-1}(n-1,d+1)+p^{(d)}_k(n-k).
\end{equation}
\end{proposition}

\begin{cor}\label{Cor3} 
The number of partitions $\lambda \vdash n$ with $k$ $d$-distant parts is equal to the sum of numbers of partitions $\mu_i \vdash n-ik+d-1$, $i=d,d+1,...$ with $k$-$1$ $d$-distant parts
\begin{eqnarray}
p_k^{(d)}(n) &=& \sum_{i \ge d} p_{k-1}^{(d)}(n-ki+d-1).
\end{eqnarray}
\end{cor}

\begin{proof}
For $q \in {\mathbb N}$, applying the Proposition \ref{Thm3} we have
\begin{eqnarray*}
p_k^{(d)}(n)&=&p_{k-1}^{(d)}(n-1,d+1)+p_k^{(d)}(n-k)\\
p_k^{(d)}(n-k)&=&p_{k-1}^{(d)}(n-k-1,d+1)+p_k^{(d)}(n-2k)\\
\ldots\\
p_k^{(d)}(n-qk)&=& p_{k-1}^{(d)}(n-qk-1,d+1)+p_k^{(d)}(n-(q+1)k)
\end{eqnarray*}
which proves the relation
\begin{equation*}
p_k^{(d)}(n)=p_{k-1}^{(d)}(n-1,d+1)+p_{k-1}^{(d)}(n-k-1,d+1)+\ldots+p_{k-1}^{(d)}(n-qk-1,d+1). 
\end{equation*}
According to the Proposition \ref{Lem3} we have to decrease the argument in every term by $d(k-1)$ in order to get the partitions without any constraint on the minimal part. This gives the identity
\begin{equation*}
p_k^{(d)}(n)=p_{k-1}^{(d)}(n-dk+d-1)+p_{k-1}^{(d)}(n-(d+1)k+d-1)+\ldots+p_{k-1}^{(d)}(n-(d+q)k+d-1). 
\end{equation*}
and proves the statement of the corollary.
\end{proof}

As an example we are going to expose that there are $12$ partitions of 18 with three 2-distant parts. According to the previous corollary we have the calculation as follows:
\begin{eqnarray*}
p^{(2)}_3(18)&=&p^{(2)}_2(13)+p^{(2)}_2(10)+p^{(2)}_2(7)+p^{(2)}_2(4)+p^{(2)}_2(1)\\
&=&5+4+2+1\\
&=&12.
\end{eqnarray*}

\begin{remark}
The previous corollary provides direct computing of the number of partitions representing the l.h.s. of the first Rogers-Ramanujan identity (when $d$=2). Having in mind that there is the one to one correspondence between these partitions and partitions having exactly one Durfee square \cite{And}, there is a more efficient computation. The similar holds for the partitions representing the l.h.s. of the second Rogers-Ramanujan identity. More precisely, it can be shown that 

\begin{eqnarray*}
p^{(2)}(n)&=& \sum_{i \ge 1} p_{\leq i}(n-i^2)\\
p^{(2)}(n,2)&=& \sum_{i \ge 1} p_{\leq i}(n-i(i+1)).
\end{eqnarray*}
\end{remark}

Note in our first example that the recurrence for $p^{(1)}_4(22)$ begins with the number of partitions of $18$ and continue with the difference 4. 
However, in our second example the first difference is $5(=18-13)$ while the other differences are equal to 3. So, for the partitions with 1-distant parts the first difference is always equal to the any other differences, for every $k$; assuming we deal with recurrence relation for the number of partition with $k$ parts expanded into terms of the number of partitions with $k$--1 parts. On the other hand, for the partitions with 2-distant parts when $k=3$ the first difference is $5$ and it increase by $2$ as $k$ increses by 1. 

Clearly, the Corollary \ref{Cor3} shows that the difference among arguments on the right hand side of the equality is always $k$ while the difference $\delta$ between the argument on the left hand side and the first argument on the right hand side depend on both $k$ and $d$. The difference $\delta$ inceases by $d$ as the $k$ increases by 1. The next table presents these first differences for the partitions with 0, 1, 2 and 3-distant parts, denoted $\delta_0, \delta_1, \delta_2, \delta_3$ respectively.
\begin{eqnarray*}
\begin{array}{rrrrr} 
k & \delta_{0} & \delta_{1} & \delta_{2} & \delta_{3} \\
3 & 1 & 3 & 5 & 7\\
4 & 1 & 4 & 7 & 10\\
5 & 1 & 5 & 9 & 13\\
\end{array}
\end{eqnarray*}

\section{Families of identities for $p(n)$}

An immediate outcome of the relation \ref{psThm1} is a nice identity between $p(n)$ and the number of partitions of $2n$ with $n$ parts. Moreover, there is an equality between $p(n)$ and the number of partitions of $mn+n$ with $mn$ parts
\begin{equation}\label{Rel39}
p(n)=p_{mn}(mn+n), \enspace m \ge1.
\end{equation}
The natural question is whether there is a similar identity between $p(n)$ and the number of partitions of a certain natural number with $1$-distant parts, and in general with $d$-distant parts. It can be shown that for $d=1$ it holds 
\begin{equation}
p(n)=p_{n}^{(1)}(2n+{n \choose 2}).
\end{equation}
This identity is generalized in the next theorem. Following the manner of our previous proofs, firstly it is shown algebraically. Than we provide a direct bijection that proves this identity. Note that, similarly, Corollary \ref{Cor3} can also be proved combinatorially.  

According to the Corollary \ref{Cor3} $p_k^{(d)}(n)$ equals either of the following sums 
\begin{eqnarray}
\sum_{i_1,i_2,\ldots,i_{k-2} \ge d} p_2^{(2)}(n-3i_1\ldots-ki_{k-2}+(k-2)(d-1))\\ 
\sum_{i_1,i_2,\ldots,i_{k-2} \ge d-1} p_2^{(2)}(n-3i_1\ldots-ki_{k-2}+(k-2)(d-1)-\frac{k(k+1)}{2}+3) \label{Rel36}
\end{eqnarray}
Similarly we have
\begin{equation}
p_k^{(d-1)}(n)=\sum_{i_1,i_2,\ldots,i_{k-2} \ge d-1} p_2^{(2)}(n-3i_1-4i_2-\ldots-ki_{k-2}+(k-2)(d-2)) \label{Rel37}
\end{equation}
The difference between arguments in the sums in \ref{Rel36} and \ref{Rel37} is $1$+$ {k \choose 2}$. Using the fact that for the number of partitions with two $d$-distant parts there is the relation
\begin{equation}
p_2^{(d)}(n)= \bigg \lfloor \frac{n-d}{2} \bigg  \rfloor, n\ge d,
\end{equation}
we get 
\begin{equation}
p_k^{(d-1)}(n)=  p_k^{(d)}(n+ {k \choose 2}) .
\end{equation}
Applying this equality to the relation \ref{Rel39} for $m=1$ completes the following statement.

\begin{thm}\label{Thm4}
The number of partitions $\lambda \vdash n$ is equal to the number of partitions $\mu \vdash 2n+d {n \choose 2} $ with $n$ $d$-distant parts
\begin{equation}
p(n)=p_{n}^{(d)}(2n+d{n \choose 2}).
\end{equation}
\end{thm}

\begin{proof}
Let $n= \lambda_1+...+\lambda_n$, $\lambda_n \ge 0$ be a parition of $n$. We present a bijection between this partition and partition of $2n+\binom{n}{2}d$ having $n$ parts. Obviously, it holds
$$2n=(\lambda_1+1)+...+(\lambda_n+1).$$

In the same manner we add $(n-1)d,...,2d,d,0$ to parts $\lambda_i$, respectively, keeping the distant $d$ among parts. The fact that resulting partition has weight $2n+\binom{n}{2}d$ and consists of $n$ parts, completes the proof. 
 
\end{proof}

Apparently, the statement of Theorem \ref{Thm4} holds for any other natural number $m$ i.e.,
\begin{eqnarray}
p(n)=p_{mn}^{(d)}(mn+n+d{n \choose 2}), \enspace m\ge1.
\end{eqnarray}
In other words, for every $n,d \in {\mathbb N}_0$ there are infinitely many sets of partitions whose length is a multiple of $n$ and whose parts are $d$-distant, that are equinumerous to $p(n)$.

\section*{Acknowledgement}
This work started during the first author's  stay at the Isaac Newton Institute in Cambridge. He thanks the Institute's personnel for their hospitality.

{\it Contacts:}\\ \\
\noindent Ivica Martinjak, \\ www: http://imartinjak.wordpress.com\\ 

\noindent Dragutin Svrtan,\\ 
E-mail address: dsvrtan@math.hr

\end{document}